\documentclass{svjour3}
\usepackage{etex} 
\usepackage[utf8]{inputenc}
\usepackage{subfig}

\usepackage{hyperref}
\usepackage{algorithm} 
\usepackage{algpseudocode} 
\usepackage{bookmark}
\usepackage{adjustbox}
\usepackage{amsmath}
\usepackage{amssymb}
\usepackage{mathtools}
\usepackage{mathrsfs}
\usepackage{comment}
\usepackage{booktabs}
\usepackage{subfig}

\usepackage{tikz}
\usetikzlibrary{shapes,decorations}
\tikzstyle{every picture}+=[remember picture]
\usepackage{pgfplots}

\DeclareMathOperator{\Prob}{Prob}
\newcommand{\NP}{\ensuremath{\mathcal{NP}}}

\newcommand{\R}{\mathbf{R}}

\newcommand{\E}{\mathbb{E}}

\newcommand{\Crit}{\mathcal{C}}
\newcommand{\MTP}{\mathcal{M}}

\definecolor{truered}{rgb}{1,.0,.0}

\definecolor{trueblue}{RGB}{0,0,255}

\definecolor{truegreen}{RGB}{0,128,0}

\newcounter{mynotes}
\ifnum\pdfshellescape=1

\else

\fi

\usepackage{datetime}

\title{Scenario Aggregation using Binary Decision Diagrams for
  Stochastic Programs with Endogenous Uncertainty}
\titlerunning{Scenario Aggregation using BDDs}

\author{%
  Utz-Uwe Haus 
  \and 
  Carla Michini
  \and
  Marco Laumanns
}
\date{}
\renewcommand\makeheadbox{}

\institute{
  {Cray EMEA Research Lab, Cray Switzerland, Hochbergerstr. 60C, 4057 Basel
    Switzerland 
    \email{uhaus@cray.com}}
  \and
  {Wisconsin Institute for Discovery, University of Wisconsin --
    Madison, 330 North Orchard Street,
    Madison WI 53715, USA
    \email{michini@wisc.edu}}
  \and
  {IBM Research, 8803 Rueschlikon, Switzerland,
  \email{mlm@zurich.ibm.com}}}

\begin{document}
\maketitle   
\begin{abstract}
  Modeling decision-dependent scenario probabilities in stochastic
  programs is difficult and typically leads to large and highly
  non-linear MINLPs that are very difficult to solve.  In this paper,
  we develop a new approach to obtain a compact representation of the
  recourse function using a set of binary decision diagrams (BDDs)
  that encode a nested cover of the scenario set.  The resulting BDDs
  can then be used to efficiently characterize the decision-dependent
  scenario probabilities by a set of linear inequalities, which
  essentially factorizes the probability distribution and thus allows
  to reformulate the entire problem as a small mixed-integer linear
  program.  The approach is applicable to a large class of stochastic
  programs with multivariate binary scenario sets, such as stochastic
  network design, network reliability, or stochastic network
  interdiction problems.  Computational results show that the
  BDD-based scenario representation reduces the problem size, and
  hence the computation time, significant compared to previous
  approaches.

\end{abstract}


\section{Introduction}
\label{sec:intro}

When modeling a stochastic optimization problem as a stochastic
program, one usually assumes that the scenario probabilities are
fixed and given input the problem. In this setup, decisions can
influence the outcome in each scenario but not their probability of
realizing. 
This standard way of reflecting the effect of uncertainty in the model
has been referred to as \emph{exogenous uncertainty}, while the term 
\emph{endogenous uncertainty} has been introduced to describe
situations where decisions can actually influence the stochastic
process itself and not only its outcomes \cite{jonsbraten:98}.

While endogenous uncertainty is straightforward to express in the
framework of Markov decision processes, its use in stochastic
programming has remained very rare, because the resulting models
appear very cumbersome to formulate and notoriously hard to solve.  
\cite{goel-grossmann:04,goel-grossmann:06} distinguish two types of
endogenous uncertainty, according to whether the decisions can
influence the temporal unfolding of events or the probability
distribution. Models of the first type are usually formulated by
enumerating the various scenario trees that correspond to the
different possible sequencing of events and couple them via
disjunctive formulations of non-anticipativity constraints (see e.g. \cite{colvin-maravelias:08}. 
Solution approaches include implicit enumeration schemes \cite{jonsbraten:98},
branch-and-bound coupled with Lagrange duality~\cite{goel-grossmann:06}, 
branch-and-cut~\cite{colvin-maravelias:10}, Lagrange relaxation of the
non-anticipativity constraints \cite{gupta-grossmann:11},
decomposition techniques \cite{gupta-grossmann:14} or
heuristics \cite{held-woodruff:05}.

In this paper we focus on problems of the second type of endogenous
uncertainty, where decisions can influence the probability
distribution of the scenarios. 
The straightforward way to model this involves products or other
non-linear expressions of decision variables thus leads to highly
non-linear models that are very hard to solve
\cite{peeta-salman-gunnec-viswanath:10,kawas-laumanns-pratsini:13}.
\cite{ahmed:00} as well as \cite{flach-poggi:10} have used
convexification techniques to deal with these polynomials, while
\cite{peeta-salman-gunnec-viswanath:10} have relied on linear
approximations.
\cite{schichl-sellmann:15} solve the non-linear stochastic program
directly with a new constraint programming approach using new type of
extreme resource constraint in combination with an efficient
propagation algorithm. 

In order to deal with the usually exponentially large scenario sets,
most approaches had to resort to scenario sampling and work with a
representative subset of scenarios. 
Recently \cite{prestwich-laumanns-kawas:13} and \cite{laumanns:14} 
proposed an exact scenario bundling approach, where scenarios that are
equivalent with respect to the recourse function are merged. 
In \cite{schichl-sellmann:15} scenarios are grouped according to the
shortest surviving paths in the underlying network.   

An important class of stochastic optimization problems that typically  
include endogenous uncertainty comes from the related areas of
stochastic network design, network reliability, and network
interdiction. Common to these problems are an underlying graph whose
edges or nodes are subject to random failures. The decisions to be
made are in order to change their failure (or survival) probabilities
of edges or nodes such that the resulting failure process has some
favorable properties, such as connectivity, shortest path lengths or
other performance indicators related to the function or process which
the graph expresses. Such problems can naturally be formulated as
two-stage stochastic programs, where those tactical or ``design''
decisions constitute the first stage, the failure process is
represented by a (usually exponentially) large set of scenarios, and
the resulting performance of the network in each failure scenario is
captured by a recourse function (which might involve auxiliary
second-stage decisions, e.g., for determining flows or shortest
paths).

We consider a general setting where scenarios can be formally
described as random binary vectors. This holds in particular for
stochastic network design and network reliability problems described
above. In many such cases, the recourse function is monotonous with
respect to the scenario. Take for example the
shortest path length between a given pair of nodes in a graph: if
scenarios indicate a set of surviving edges after some disaster, then
the shortest path length can never increase when fewer edges
 survive. 
Our main idea is exploit this structure for scenario aggregation, and
our main contribution is a new method for scenario aggregation by
means of binary decision diagrams. 
The approach not only enables a substantial reduction of the
(initially exponentially large) scenario set without any accuracy
loss, but also to efficiently characterize the decision-dependent
scenario probabilities by a set of linear inequalities. This
essentially factorizes the probability distribution and thus allows to
reformulate the entire problem as a small mixed-integer linear
program.  

Our motivation and running example will be the shortest path problem
hinted at above:
\begin{example}\label{ex:pre-disaster-planning}
  Let $G=(V,E)$ be a graph whose edges have some (independent)
  failure probabilities $p_e\in[0,1]$ ($e\in E$). After a disastrous
  event some edges will have failed, and each failure scenario is
  described by a set of surviving edges $\xi\subseteq E$. We are
  interested answering the following two questions:
  \begin{enumerate}
  \item What is the expected value of the shortest path
  lengths between a set of distinguished nodes? 
  \item If we can take some actions to modify the failure
    probabilities $p_e$, how should this be done optimally in order to
    to minimize the expected shortest path length?
  \end{enumerate}
\end{example}
If the graph is a road network, and edges fail due to an earthquake,
this is the setting of~\cite{peeta-salman-gunnec-viswanath:10}. The
set of all failure scenarios can be totally ordered by comparing the
lengths of the shortest path between two (fixed) nodes: If in 
failure scenario there exists an $st$-path of length
at most $\alpha$, then in all scenarios where only a subset of these
edges fails, this path still exists. Scenarios with the same shortest
path length will be considered equal in the total order. Of course,
depending on the application, 
many other scenario orderings are conceivable, e.g., according to longest
$st$-path, number of disjoint $st$-paths, or size of the largest
connected component of the graph. 

The successive shortest path problem is closely related to the classic
$st$-reliability problem introduced by 
\cite{moore-shannon:56a,moore-shannon:56b} and 
\cite{birnbaum-esary-saunders:61}
and has been studied by many authors in numerous variants since then. Here the
network is a system of circuits with edges corresponding to components
which have some probability of failing. The system as a whole operates
if there exists some $st$-path, and the probability of having such
a path is called $st$-reliability.
This fundamental problem and its various extensions have a vast number
of applications in reliability engineering~\cite{billinton-allan:92}.

The question also naturally arises in the setting of interdiction
problems, where typically one considers decisions to be actions of an
attacker to weaken the network structure (to increase failure
probability of edges, reduce capacity, etc.) in order to hamper the
network's operational capability. For a survey on typical network
interdiction problems and corresponding modeling and solution
approaches see~\cite{cormican-morton-wood:98}. 

The approach we develop in this paper is not limited to combinatorial stochastic optimization
problems. It can also be used in computing the expected objective
function value for linear programs with varying right-hand-side
coefficients, where each coefficient independently attains one of two values
with a given probability.
Consider a maximization problem with less than or equal constraints, and assume that: $(i)$ in the nominal problem all right-hand-side coefficients are at the larger of the two values; $(ii)$ the
scenarios are given by sets of rows where the coefficient takes the
smaller value.
Then the objective function is monotonously decreasing
when taking supersets among the scenarios. In a similar way, irreducible inconsistent linear systems and maximal consistent subsets of LPs~\cite{parker-ryan:96,amaldi-pfetsch-trotter:03} can be studied.

\section{Problem Setting}
\label{sec:setting}

We consider a two-stage stochastic optimization problem
\begin{equation}
  \label{eq:2-stage-stoch}
  \begin{aligned}
    \min &\,\E_{\xi|x}[f(\xi)]\\
    Cx&\leq d\\ 
    x_e&\in\{0,1\}&&e\in E\\
    \xi&=(\xi_e)_{e\in E}\in\{0,1\}^{|E|}
  \end{aligned}
\end{equation}
where we minimize the expected value (over all realized failure
scenarios) of the recourse function $f(\xi):2^{E}\to\R$ conditioned on
first-level
decisions $x_e$, where $e\in E$ has an associated elementary failure
event $\xi_e$ whose probability may be 
influenced by decision $x_e$. The probability distribution of the
scenario distribution thus changes 
depending on the first-level decisions. The first-level decisions are
to be taken subject to some set of linear constraints, e.g., a budget
restriction $\sum_{e\in E}x_e\leq \beta$. Typically the evaluation of
the recourse function $f(\xi)$  will itself amount to solving an optimization
problem.  We will interpret scenarios
as sets of `positive events', e.g., the sets of surviving edges in a
network after some disaster.

In order to make our assumption on the problem structure more precise,
we define the following properties and objects related to the recourse
function. 

\begin{definition}[aggregable recourse function]
  The recourse function $f$ of a 2-stage stochastic optimization
  problem is called 
  \emph{aggregable} if
  \begin{enumerate}
  \item $f$ does not depend on the first-stage decision $x$,
  \item $f$ is counter-monotone with taking subsets of scenarios,
    \begin{equation}
      \xi_1\subseteq \xi_2\Rightarrow f(\xi_1)\geq f(\xi_2)\label{eq:f-countermonotone}
    \end{equation}
    for all $\xi_1,\xi_2\in 2^{E}$,
  \item the probabilities of events $e\in E$ are independent.
  \end{enumerate}
  
  We denote the \emph{critical values} of $f$, i.e.~the different values $f$
  takes, by $\Crit(f)=\{\alpha\::\:\alpha=f(\xi),\xi\in 2^{E}\}$.
  Moreover we denote by
  $\mathcal{T}\MTP_\alpha(f) = \{ \bar{\xi} \in 2^E : f(\xi) > \alpha,\; \xi = 1^E - \bar{\xi}\}$
  the sets of edge failures that forces $f$ to take values strictly greater than $\alpha$.
  Finally we define the set of minimal survivable scenarios for each critical value by
  $\MTP_\alpha(f)=\{\xi\in 2^{E}\::\:f(\xi)=\alpha, \forall
  \xi'\subset\xi: f(\xi')>f(\xi)\}$, with $\alpha \in \Crit(f)$.
\end{definition}



The elements of $\MTP_{\alpha}(f)$ are the transversal of the clutter given by the minimal members of $\mathcal{T}\MTP_\alpha(f)$.

\addtocounter{example}{-1}
\begin{example}[continued]
  In our running example, computation of $f$ amounts to computing a
  shortest path length $f_{\text{SP}}$ in a graph whose edge set
  depends on the scenario: the edges that failed are missing in the
  scenario compared to the initial graph $G$. Clearly, shortest path
  length is counter-monotone with removing supersets of edges,
  i.e., considering subsets of the surviving edges of some other
  scenario. Another way of looking at this is to say that the sublevel
  sets $L^-_{\alpha}(f)=\{\xi\::\: f(\xi)\leq \alpha\}$ of $f$ are
  co-monotone (wrt. inclusion) with the natural ordering of real
  numbers $\alpha_i<\alpha_j$.  Each set $\MTP_\alpha(f)$ consists of
  the simple paths of length $\alpha$ in the graph.
\end{example}

The set $\Crit(f)$ allows us to rewrite the objective function as
follows:
$$\E_{\xi|x}[f(\xi)]=\sum_{\alpha\in\Crit(f)}\alpha\Prob[f(\xi)=\alpha|x].$$
We are thus interested in computing the probabilities
$\Prob[f(\xi)=\alpha]$, first without taking into account possible
decision variables $x_e$, and then integrating them.

\section{Aggregation of Nested Scenario Covers}
\label{sec:scenario-aggregation}

In this section, we develop our approach to provide an effective
scenario aggregation by using nested scenario covers. We first
describe how scenario sets can be encoded as BDDs in such a way that
these sets represent the sublevel sets of the recourse function. We
then proceed to show how scenario probabilities can be represented as
a flow of probability on the DAGs represented by the BDDs. Finally, we
describe how those probability flows can be shaped using the binary
first-stage decision variables via linear inequalities and a derive
the resulting MIP formulation.

\subsection{Scenario Aggregation using BDDs}
\label{sec:scenario-bundles}

We will from now on consider the set $\Crit(f)$ as an ordered set
$(\alpha_0,\dots,\alpha_N)$ with the natural ordering $<$ of the real
numbers. This induces an ordering of the set $\MTP(f)$ as
$(\MTP_{\alpha_0}(f),\dots,\MTP_{\alpha_N}(f))$. 

Each $\MTP_\alpha(f)$ induces a monotone Boolean function
$\Phi^\leq_\alpha$ on the scenarios whose minimal true points are the
members of $\MTP_\alpha(f)$ by ``$\Phi^\leq_\alpha(\xi)=1$ if and only
if $f(\xi)\leq\alpha$.'' It is well-defined because
of~\eqref{eq:f-countermonotone}. Note that $\Phi^\leq_\alpha(\xi)$
describes the support of the cumulative distribution function
$\Prob[f(\xi)\leq\alpha]$. We refer to~\cite{crama-hammer-book} for
more details on Boolean functions. In the reliability analysis setting
the function $\Phi^\leq_\alpha$ is called system state function, and
we can assume that the system is a coherent binary system. Then
$\MTP_\alpha(f)$ are the path sets of the system and
$\Prob[\Phi^\leq_\alpha(\xi)=1]$ is the reliability, see~\cite{ball-provan:88}.

We propose to encode each function $\Phi^\leq_\alpha$ for
$\alpha\in\Crit(f)$ in the form of a (reduced, ordered) binary
decision diagram (BDD), see~\cite{lee:59,bryant:86}. Intuitively, a
BDD can be understood as a directed acyclic graph with a single source
and a single sink, in which each source-sink path encodes one or more
feasible points of a Boolean function. The graph is layered, with the
layers indexed by the variables of the function. The key property is
that every node has at most 2 children, and for each sub-dag that
includes a sink there are no isomorphic copies in the BDD. This can
make them much more compact than conventional decision trees encoding
the same feasible points. Although BDDs can be of exponential size
compared to the function they encode, for various classes of functions
one can find compact BDD encodings.  Note that finding a minimal size
BDD is $\NP$-hard, even for monotone Boolean functions
\cite{takenaga-yajima:00}, and not every BDD construction algorithm
will be output-polynomial.  Despite all of these caveats, BDDs have
been highly successful in many applications,
see~\cite{meinel-theobald:98} for an overview. Furthermore, good
software for BDD handling exists, which provides various heuristic BDD
size minimization strategies, e.g, \texttt{cudd}
of~\cite{somenzi:cudd-2.5.0}.

Recently, it was shown that BDDs encoding independence systems can
always be constructed in output-polynomial time, if equivalence of
minors in the associated circuit system can be checked efficiently by
a top-down compilation rule of~\cite{haus-michini:14}. We will use
this to prove Lemma~\ref{lem:something-with-guaranteed-size}.

\begin{definition}[BDD]
  Let $E=(e_1,\dots,e_{|E|})$ be a finite linearly ordered set. A
  \emph{binary decision diagram} (BDD)
  $B=(E,V,A\subseteq V\times
  V,\epsilon:V\to \{1,\dots,|E|\},l((u,v)):A\to\{0,1\})$ is either degenerate,
  $B=(E,\{\bot\},\emptyset,\epsilon,l)$ or
  $B=(E,\{\top\},\emptyset,\epsilon,l)$, or consists of a directed
  acyclic graph on node set $V$ with exactly one root $u^*$, two
  leaves $\top,\bot$, arc set $A$, arc label function $l:A\to\{0,1\}$
  and a layering function
  $\epsilon:V\to \{1,\dots,|E|\}$.
  Unless $B$ is
  degenerate, it must satisfy the following conditions:
  \begin{itemize}
  \item
    $B$ is a layered digraph, i.e. its node set partitions into layers
    $L_i = \{u \in V \::\: \epsilon(u)=i\}$, $i = 1,\dots,|E|$ and 
    $L_{|E|+1}=\{\top\}$, such that $|L_{1}|=1$ (the root layer) and
    $L_{|E|+1}$ is the terminal layer.
  \item Arcs only extend to layers with higher index, i.e.
    $\forall (u,v)\in A, \; \epsilon(u) < \epsilon(v)$.
  \item Every node $u\in V\setminus\{\bot,\top\}$ is the tail of exactly two differently
    labeled arcs,
    i.e. $\forall u\in V\setminus\{\bot,\top\}\exists v_0,v_1\in V:v_0\neq v_1,(u,v_0)\in
    A,(u,v_1)\in A,l((u,v_0))\neq l((u,v_1))$.
  \item For any two distinct nodes $u,v$ of $V$, the sub-BDDs rooted at
    $u$ and $v$ are not isomorphic, i.e. $B_u\not\cong B_v$.\\
    Here $B_u$ denotes the sub-BDD of $B$ defined by
    the sub-dag of $(V,A)$ rooted at $u$, with node and arc label
    functions obtained by restriction of $\epsilon$ and $l$ to $V(B_u)$ and $A(B_u)$, respectively. BDD-isomorphism $\cong$ is defined
    as isomorphism of directed graphs with identical variable set and
    identically evaluating functions $\epsilon$ and $l$.
  \end{itemize}
  For each node $u\in V$, the outgoing arc labeled by $1$ is called
  the \emph{\textsc{True}-arc}, and the outgoing arc labeled by $0$ is
  called the \emph{\textsc{False}-arc}. 
\end{definition}
Note that BDDs according to this definition are called \emph{reduced,
  ordered BDD} in the classical
literature~\cite{lee:59,bryant:86,meinel-theobald:98}. (In drawing
BDDs one often suppresses the $\bot$-node and all arcs
$\delta^{\text{in}}(\bot)$, since these can be easily reconstructed.)

For $i = 1,\dots,|E|$, the \emph{layer width} $w_i$ of layer $L_i$ is defined as $w_i=|L_i|$. The \emph{BDD-width} is then
$w = \max_{i = 1,\dots,|E|} w_i$. The \emph{total size} of a BDD is $|V|+1=\sum_{i = 1}^{|E|+1} w_i$.

Each BDD $B$ represents a Boolean function: 
\begin{equation}
  \label{eq:bdd-interpretation}
  \Phi_B(x)=\bigvee_{\substack{%
      \text{$P$ an}\\
      \text{$u^*$-$\top$ path}}}
  \left(\bigwedge_{\substack{%
        \text{$(u,v)$ edge of}\\
        \text{$P$}:l((u,v))=1}}
    x_{\epsilon(u)}
    \wedge
    \bigwedge_{\substack{
        \text{$(u,v)$ edge of}\\
        \text{$P$ :l((u,v))=0}}}
    (\lnot x_{\epsilon(u)})\right).  
\end{equation}

Furthermore, every Boolean function has a BDD representation (which is essentially unique given the ordering of $E$).

\begin{remark}\label{rem:dual-bdd}
  Let $\Phi(x)$ be a monotone Boolean function and
  $\Phi^D(x)=\lnot\Phi(\lnot x)$ its dual. If $B=(E,V,A,\epsilon,l)$
  is a BDD representing $\Phi(x)$, then $B'=(E,V,A,\epsilon,l')$ with
  $l'((u,v))=1-l((u,v))$ represents $\lnot\Phi^D(x)$, i.e. except for
  the labeling, $\Phi(x)$ and $\Phi^D$ have the same BDD.
\end{remark}

This is a direct consequence of the fact that the BDD not only encodes
all feasible points of the monotone Boolean function but also all
infeasible points (as paths from $u^*$ to $\bot$) and
formula~\eqref{eq:bdd-interpretation}.

The isomorphism between BDDs for a monotone Boolean function and its
dual function is useful if the set $\mathcal{M}_\alpha$ are given
as a list of minimal true elements: this list provides a covering-type
formulation for the dual function and the output-polynomial time top-down
BDD construction rule using matrix minors of~\cite{haus-michini:14}.

\subsection{Integrating Scenario Probabilities}

\begin{lemma}
  Given a BDD $B$ encoding a Boolean function $\Phi:2^E\to\{0,1\}$ and
  a scenario distribution such that the events $e\in E$ are
  independently distributed random variables, the value
  $\Prob[B]:=\Prob[\Phi(\xi)=1]$ can be computed in linear time.
\end{lemma}

Note that linear time here refers to linear in the size of the BDD $B$ (and $|E|$); in the proof we only have to show that we can obtain a system of equations computing the probability that is shaped like the BDD and does not have excessively large subexpressions.

\begin{proof}
  The construction is similar to the classical computation of
  probabilities in a scenario tree, with the difference being that
  recurring computations in subtrees are avoided because the BDD
  structure (`no isomorphic sub-BDDs') merges these cases.\\
  Let 
  $p_i\in [0,1]$ be the probability of event $e_i\in E$. We proceed
  by induction: If $B$ is the BDD consisting of only the
  $\top$-terminal, it encodes the entire cube $\{0,1\}^{|E|}$ and
  $\Prob[\Phi(\xi)=1]=1$. Similarly, if it consists only of the $\bot$
  terminal, in encodes the empty set, so $\Prob[\Phi(\xi)=1]=0$.\\
  Assume now that the statement has been proven for BDDs with $k-1>0$
  layers. Let $B$ be a BDD with $k$ layers. The unique root node
  $u^*$ has exactly 2 children that we'll call \texttt{True}-child and \texttt{FALSE}-child; let the event for the root layer
  be $\epsilon(u^*)=e_1$.
  \\
  Initially assume that the sub-BDDs have their roots in the layer
  $L_2$ directly following $e_1$. Then,
  denoting the probabilities computed for the sub-BDDs by $p_{\text{\textsc{True}-child}}$ and $p_{\text{\textsc{False}-child}}$,
  \begin{equation}
    \label{eq:bdd-prob-recursion/2}  
    \Prob[\Phi(\xi)=1]=p_{1}p_{\text{\textsc{True}-child}}+(1-p_{1})p_{\text{\textsc{False}-child}},    
  \end{equation}
  since scenarios contain/don't contain event $e_1$ with these
  probabilities and the probabilities of the completions have been
  computed inductively for the sub-BDDs of size at most $k-1$.

  If a sub-BDD, say the one for $\xi_1=1$, has its root in a
  layer $l>2$, then this sub-BDD encodes the set of solutions
  $(1,\underbrace{\star,\dots,\star}_{l-2},\xi')$ where $\star$ denotes that $0$
  or $1$ can be chosen arbitrarily for every 
  $(\xi')\in\{0,1\}^{|E|-l}$.
  For simplicity assume
  that only the \textsc{True}-child of $u^*$ starts at layer $l$. Then
  \begin{equation}
    \begin{aligned}
       \Prob[\Phi(\xi)=1]&=p_{1}\left(\prod_{i=2}^{l-1}(p_{i}+(1-p_{i}))\right)p_{\text{\textsc{True}-child}}+(1-p_{1})p_{\text{\textsc{False}-child}}\\
    &=p_{1}p_{\text{\textsc{True}-child}}+(1-p_{1})p_{\text{\textsc{False}-child}},      
    \end{aligned}
    \label{eq:bdd-prob-recursion/1}
  \end{equation}
  since the scenarios do not depend on events on layer
  $2,\dots,l-1$ for the choice of $\xi_1=1$.
  
\end{proof}

\subsection{Shaping the Scenario Distribution}

Equations~\eqref{eq:bdd-prob-recursion/2}
and~\eqref{eq:bdd-prob-recursion/1} give a recursive set of linear
equations (starting with $p_\top=1, p_\bot=0$ for the leafs) that
can be used to directly turn a BDD encoding $\MTP_\alpha(f)$ into a
set of linear constraints to compute $\Prob[\Phi^\leq_\alpha(\xi)=1]$
using as many variables and equations as the BDD has nodes (each node
has exactly one defining equation). We now show
how decisions influencing the probabilities $p_i$ can be incorporated
into these formulas as well.

\begin{lemma}
  Given a BDD $B$ encoding a Boolean function $\Phi:2^E\to\{0,1\}$ and
  a scenario distribution such that the events $e_i\in E$ are
  independently distributed random variables depending on
  decisions $(x_i)_{i=1,\dots,|E|}$, the value $\Prob_{\{\xi|x\}}[\Phi(\xi)=1]$ can be
  computed in linear time.  
\end{lemma}

\begin{proof}
  This follows immediately from~\eqref{eq:bdd-prob-recursion/1}
  because the events are assumed
  to be independent: Let $p_i(x)$ denote the conditioned
  probabilities, and
  $p_{\text{\textsc{True}-sub}}(x)$, $p_{\text{\textsc{False}-sub}}(x)$
  the conditioned sub-BDD probabilities. Then
  \begin{equation}
    \Prob[\Phi(\xi)=1]=p_{1}(x)p_{\text{\textsc{True}-sub}}(x)+(1-p_{1}(x))p_{\text{\textsc{False}-sub}}(x),\label{eq:bdd-condprob-recursion/2}
  \end{equation}
  just as in~\eqref{eq:bdd-prob-recursion/2}.
\end{proof}

Again, equation~\eqref{eq:bdd-condprob-recursion/2} gives a recursive
set of equations (starting with $p_\top=1,p_\bot=0$ for the leafs),
but depending on how $p_i(x)$ is defined they may not be linear. In
the next section we will show that if the decisions are binary (or can
take only a fixed number of values), the problem can be formulated as
a compact mixed-integer programming problem.

\subsection{MIP Model for Aggregated Scenario Probabilities}
\label{sec:mip-model-scenario}

Let us now consider for ease of presentation the case where $|E|$
binary decision variables $x_i\in\{0,1\}$ influence the nominal event
probability $p_e$ such that 
$$p_i(x)=
\begin{cases}
  p_i & \text{if } x_i=0\\
  p_i+\Delta_i & \text{if } x_i=1
\end{cases}$$
where $\Delta_i\in[-p_i,1-p_i]$ is the boost or decrease
of the probability of event $e_i\in E$ when taking decision $x_i=1$.

With this definition, and in light of
equation~\eqref{eq:bdd-condprob-recursion/2}, we define for each arc
$(u,v)$ in the BDD $B^\leq_\alpha=(E,A,\epsilon,l)$
$$p_{(u,v)}(x)=
\begin{cases}
  p_i(x) &     \text{if }(u,v)\in A, \epsilon(u)=i, l((u,v))=1\\
  (1-p_i(x)) & \text{if }(u,v)\in A, \epsilon(u)=i, l((u,v))=0,
\end{cases}
$$
that is, depending on the decision $x_i$ all arcs in the BDD that
start at node $u$ associated with event $e_i$ are assigned the
appropriate probability depending on $x_i$ and whether they correspond
to the event being in the scenario ($l((u,v))=1$) or not
($l((u,v))=0$). Introducing intermediate variables for each BDD node
we can thus formulate the following MIP constraints to compute the
probability of all scenarios of $B^\leq_\alpha$ subject to decision variables
$x=(x_i)_{i=1,\dots,|E|}$:

\begin{equation}\label{eq:P-alpha-block}
  P_\alpha=\left\{
    (p_\alpha^\leq,x)\::\:
    \begin{aligned}
      p_\alpha^\leq(u) &= 1 && u=\top\in V(B^\leq_\alpha)\\
      p_\alpha^\leq(u) &= 0 && u=\bot\in V(B^\leq_\alpha)\\
      p_\alpha^\leq(u) &= p_{(u,v)}(x) p_\alpha^\leq(v) + p_{(u,w)}(x) p_\alpha^\leq(w)
      &&\\
      &&&\hspace{-4.5em}(u,v),(u,w)\in A(B^\leq_\alpha), v\neq w\\
      p_\alpha^\leq&\in [0,1]^{V(B^\leq_\alpha)}\\
      x&\in \{0,1\}^E
    \end{aligned}
    \right\}
\end{equation}

It is easy to see that the constraints of $P_\alpha$ can be written as
linear inequalities, which we avoided above to unclutter the
formulation. For example, in the case $(u,v),(u,w)\in
A(B^\leq_\alpha), v\neq w, l((u,v))=1, l((u,w))=0$ where $\epsilon(u)=i$:
$$ 
\begin{aligned}
  p_\alpha^\leq(u) &= p_{(u,v)}(x) p_\alpha^\leq(v) + p_{(u,w)}(x) p_\alpha^\leq(w)\\
  \text{can be written using inequalities}\\
  p_\alpha^\leq(u) &\leq (p_i+\Delta_i)p_\alpha^\leq(v) + (1-p_i-\Delta_i)p_\alpha^\leq(w) + (1-x_i)\\
  p_\alpha^\leq(u) &\leq  p_i          p_\alpha^\leq(v) + (1-p_i)         p_\alpha^\leq(w) + x_i\\
  p_\alpha^\leq(u) &\geq (p_i+\Delta_i)p_\alpha^\leq(v) + (1-p_i-\Delta_i)p_\alpha^\leq(w) - (1-x_i)\\
  p_\alpha^\leq(u) &\geq  p_i          p_\alpha^\leq(v) + (1-p_i)         p_\alpha^\leq(w) - x_i.
\end{aligned}
$$

Then the formulation of $P_\alpha$ contains $|V(B^\leq_\alpha)|+|E|$
variables and $4|V(B^\leq_\alpha)|+1$ constraints (plus box constraints
for all variables). 

Clearly, similar models can be built if $x_i$ is allowed to take a
finite number of discrete values, using $2$ constraints per possible
value of $x_i$, or disjunctive, or constraint programming techniques.

To obtain a MIP model for aggregable problems of the
form~(\ref{eq:2-stage-stoch}) we can simply combine blocks of the
form~\eqref{eq:P-alpha-block} using inclusion-exclusion on the scenario sets, exploiting:
\begin{equation}
  \label{eq:2-stage-stoch-agg-mip}
  \begin{aligned}
    \min_{\alpha\in\Crit(f)} &\alpha p^=_\alpha\\
  p^=_{\alpha_0}&=p^\leq_{\alpha_0}(u^*) & u^*\text{ root node of }B^\leq_{\alpha_0}\\
  p^=_{\alpha_i}&=p^\leq_{\alpha_i}(u^*) - p^=_{\alpha_{i-1}} & u^*\text{ root node of }B^\leq_{\alpha_i}, \alpha_i\in\Crit(f)\setminus\{\alpha_0\}\\
  Cx&\leq d\\
  (p^\leq_{\alpha_i},x)&\in P_{\alpha_i} & \alpha_i\in\Crit(f)
\end{aligned}
\end{equation}
Note that the $x$ variables will be shared between $P_{\alpha_i}$
and $P_{\alpha_j}$ for $\alpha_i,\alpha_j\in\Crit(f)$.

When can we expect a polynomial-size MIP formulation using aggregated 
scenarios? This depends on two prerequisites: We need to be sure that
the set $\Crit(f)$ and the BDDs $B^\leq_\alpha$ are small. The first
condition may be satisfied automatically for some problem classes, or
may be a consequence of the way that the instance is encoded, e.g.,
giving an explicit list of relevant objective values. The second
condition is hard to control in general, but often one can identify
parametric problem subclasses such that the width of $B^\leq_\alpha$
is appropriately bounded for each fixed value of the parameter: Bounds
on the BDD width have been studied widely. For our purposes we will
use the results of~\cite{haus-michini:14} and~\cite{pan-vardi:05}.

To actually construct the formulation we also need to ensure that the
construction of each $B^\leq_\alpha$ can be performed in
(output-)polynomial time. The easiest way to ensure this is to specify
a top-down-compilation rule, such that each BDD node is touched only
once in the construction, and decisions about whether to merge the
child subtrees are made immediately in the node in polynomial
time. For monotone Boolean functions, which encode the members of an
independence system, this can be achieved if equivalence of minors of
the associated circuit system can be checked efficiently,
see~\cite{haus-michini:14}. A prominent example is that of the graphic
matroid~\cite{sekine-imai-tani:95}, i.e. when scenarios are forests or
simple cycles of a graph. By Remark~\ref{rem:dual-bdd} this is also
always the case if the sets $\mathcal{M}^\leq_\alpha$ (or the dual
sets) are explicitly given. Thus if $\mathcal{M}^\leq_\alpha$ is
explicitly given or can be enumerated in polynomial time, efficient
BDD construction reduces to the question whether the BDD width
is polynomially bounded.

In light of this reformulation technique we make the following
definition.

\begin{definition}[polynomially aggregable 2-stage stochastic
  optimization problem]
  An aggregable $2$-stage stochastic optimization problem is called
  \emph{polynomially aggregable} if 
  \begin{itemize}
  \item the number of critical values $\Crit(f)$ of $f$ is polynomial,
    and
  \item all BDDs $B^\leq_\alpha(f)$ can be constructed in polynomial
    time.
  \end{itemize}
\end{definition}

It follows from~\cite[Thm.~6]{haus-michini:14} that whenever the 
matrix containing the incidence vectors of the elements of
$\mathcal{T}\MTP_\alpha(f)$
has bandwidth $k$, then $B^\leq_\alpha(f)$
has size at most $n2^{k-1}$, which is polynomial if $k\in\mathcal{O}(log
n)$. Hence 2-stage stochastic optimization problems with
bandwidth-limited sets $\mathcal{T}\MTP_\alpha(f)$ and a polynomial
number of critical values $\Crit(f)$ are polynomially aggregable:

\begin{lemma}\label{lem:something-with-guaranteed-size}
  2-stage stochastic optimization problems with input size $\sigma$,
  a polynomial number of critical values $\Crit(f)$, and an
  explicitly given sets $\mathcal{T}\MTP_\alpha(f)$ whose
  incidence matrix has bandwidth
  $k\in\log(\sigma)$ are polynomially aggregable.
\end{lemma}

Note that our definition of~\eqref{eq:2-stage-stoch-agg-mip} contains
as a special case the \emph{union of products (UPP)} problem (this is
the case where the decisions $x_e$ have no influence on the
probabilities). \cite{ball-provan:88} show that UPP is $\NP$-hard,
even when the sets $\MTP(f)$ are explicitly given -- thus unless
$\mathcal{P}=\NP$ the BDD construction will be exponential in these
cases.

\section{Applications}
\label{sec:applications}

\subsection{Pre-disaster Planning Problem}
\label{sec:pre-disast-plann}
The pre-disaster planning problem we consider first is an instance of the running example~\ref{ex:pre-disaster-planning} we used
throughout.  As mentioned before, the algorithmic challenges are easy
to answer: The set of all shortest paths $\MTP(f_{\text{SP}})$ between
a pair of nodes can be enumerated in time
$\mathcal{O}((|V|+|E|)|\MTP(f_{\text{SP}})|)$ using the classic
restricted backtracking technique of~\cite{read-tarjan:75} or the
recent output-linear method of~\cite{birmele-ferreira-etal:13}. One
could alternatively enumerate the elements of
$\mathcal{T}\MTP(f_{\text{SP}})$ using the method
of~\cite{provan-shier:96}, which may be preferable if one is
interested in displaying the minimal sets of simultaneously failing
edges characterizing a bundle of scenarios to a user.  Actually, since
$\mathcal{T}\MTP_\alpha(f_{\text{SP}})$ may be exponential in the size
of $\MTP_\alpha(f_{\text{SP}})$ (a graph consisting of $k$
edge-disjoint $st$-paths with at most $\alpha$ edges each has $\alpha^k$
minimal sets of edge failures defining $\MTP_\alpha(f_{\text{SP}})$),
it may be preferable to start computing
both sets concurrently, and stop when one is complete. From both sets
one can construct identically-sized BDD (which differ only by the arc
labels), since the sets define a pair of dual monotone Boolean
functions.

\begin{table}
  \centering
  \subfloat[][With path length cutoffs, penalty $120$ for excessively
  long paths/unconnectedness.]{
    {\small\def~{\hphantom{0}}
    \begin{tabular}{lc@{\quad}cc@{\quad}cc}
    \toprule
    O-D-pair &Dist.-Limit & \#bundles & MIP size & \# BDDs & MIP size\\
                            &&
                            \multicolumn{2}{c}{\small(in~\textsc{Prestwich
                              '13})}
                            &
                            \multicolumn{2}{c}{\small(using BDD-bundles)}\\
    \midrule
    14--20 &     31 & 39  &                  &  4 & $~237 \times  89~$\\ 
    14--7  &     31 & 29  &                  &  6 & $~333 \times 113$\\
    12--18 &     28 & 56  &                  &  4 & $~237 \times  89~$ \\ 
     9--7  &     19 & 26  &                  &  4 & $~164 \times  71~$\\ 
     4--8  &     35 & 73  &                  &  6 & $~421 \times 135$\\ 
     \midrule
     \multicolumn{1}{r}{$\sum$}
           &        &223 &$14174\times 6221$&  24& $ 1466\times 454$\\
     \multicolumn{1}{r}{MIP solution}
           &     &                    &    && $1$s\\
    \bottomrule
  \end{tabular}
}
}\vspace{\baselineskip}

\subfloat[][Without path length cutoffs.]{
  \centering
  {\small
    \begin{tabular}{l@{\quad}cc@{\quad}cc}
    \toprule
    O-D-pair  & \#bundles & MIP size & \# BDDs & MIP size\\
              &           
                            \multicolumn{2}{c}{\small(in~\textsc{Prestwich
                              '13})}
                            &
                            \multicolumn{2}{c}{\small(using BDD-bundles)}\\
    \midrule
    14--20 & 378 &                  &  14&  $2609 \times  682$\\ 
    14--7  & 712 &                  &  30& $13097 \times 3304$\\ 
    12--18 & 233 &                  &  8 &   $\hphantom{00}997 \times 1026$\\ 
     9--7  & 266 &                  &  8 &  $1137 \times  314$\\ 
     4--8  & 305 &                  &  12& $2301  \times  605$\\ 
     \midrule
     \multicolumn{1}{r}{$\sum$}
           & 1894&$123682\times 56851$&  72& $20137\times 5064$\\
     \multicolumn{1}{r}{MIP solution}
           &     &                    &    & $36$s\\
    \bottomrule
  \end{tabular}}}
\caption{Scenario partition bundles vs. BDD bundles in the Istanbul
    road network problem instance
    of~\cite{peeta-salman-gunnec-viswanath:10}.
    MIP size is number of
    constraints/number of rows, excluding $[0,1]$-box constraints
    before preprocessing of the solver. All
  MIPs have $30$ binary decision variables and one budget
  constraint. Total path enumeration and BDD construction time is
  $<1s$. CPLEX 12.5 in single-thread mode on a dedicated 24-core,
  4-CPU Intel X5650/2.67 GHz system with 96Gb RAM running Linux 3.14.17.}
  \label{tab:part-vs-bdd-bundles}
\end{table}

\begin{table}
  \centering
  {\small\let\mc\multicolumn
    \begin{tabular}{r@{\quad}rrr@{\quad}rrrrrr}
    \toprule
        &      &                 & \# O-D-            &min  &25\%   &median & 75\%  & 99\%  & max\\            
    $n$ & arcs & $2^{\text{arcs}}$  &       pairs &size &size   &size   &size   &size   & size\\
    \midrule                                                                 
    \multicolumn{10}{c}{$\alpha=1.1$}\\[2ex]
    1   &  5   & $32$            & 64           & 2  & 2 & 3 &  3 &   5 &     5\\
    2   & 11   & $2048$          & 384          & 2  & 3 & 4 &  5 &  14 &    26\\
    3   & 19   &$524288$         & 1280         & 2  & 3 & 5 &  6 &  33 &   126\\
    4   & 30   &$1.1\times10^9$  & 3200         & 2  & 4 & 6 &  9 &  79 &   623\\
    5   & 43   &$8.8\times10^{12}$& 6720        & 2 & 5 & 7 & 17 & 125 &   921\\
    6   & 59   &$5.8\times10^{17}$& 12544       & 2 & 6 & 9 & 26 & 258 &  2778\\
    7   & 77   &$1.5\times10^{23}$& 21504       & 2 & 6 & 11& 35 & 518 & 22967\\
    \midrule
    \multicolumn{10}{c}{$\alpha=1.5$}\\[2ex]
    1   &  5   & $32$             &64           &2 & 2 & 3&   4 &     9 &      9\\
    2   & 11   & $2048$           &384          &2 & 3 & 4&   8 &    37 &     64\\
    3   & 19   &$524288$          &1280         &2 & 4 & 6&  22 &   261 &    565\\
    4   & 30   &$1.1\times10^9$   &3200         &2 & 5 &16&  69 &  1397 &   6080\\
    5   & 43   &$8.8\times10^{12}$& 6720        &2 & 6 &26& 157 &  5062 &  24374\\
    6   & 59   &$5.8\times10^{17}$& 12544       &2 &12 &58& 398 & 22525 & 134298\\
    7   & 77   &$1.5\times10^{23}$& 21504       &2 &16 &98& 943 &122126 &1870925\\
    \midrule
    \multicolumn{10}{c}{$\alpha=\infty$}\\[2ex]
    1   &  5   & $32$             &64           &3 &    6 &     6 &     8 &      9 &      9 \\
    2   & 11   & $2048$           &384          &2 &    9 &    21 &    27 &     55 &     64\\
    3   & 19   &$524288$          &1280         &2 &   74 &   197 &   347 &    884 &   1584\\
    4   & 30   &$1.1\times10^9$   &3200         &2 &  128 &   902 &  2295 &   8218 &  16267\\
    5   & 43   &$8.8\times10^{12}$& 6720        &2 &  543 &  7336 & 20748 &  74896 & 220532\\
    6   & 59   &$5.8\times10^{17}$& 12544       &2 & 1341 & 13420 & 56704 & 679862 &1462435\\
    7   & 77   &$1.5\times10^{23}$& 21504       &2 &14311 &125457 &678671 &3680199 &8186671\\
    \bottomrule
  \end{tabular}}
\caption{BDD bundles for random road networks on grid of $(n+1)^2$
  nodes with density approximately $1.2$, see~\cite{masucci-smith-crooks-batty:09}.
  Experiments repeated for $32$ different networks and for all origin-destination pairs in each network, subject to length cutoff $\alpha\cdot d$, where $d$ is the shortest path distance
  for the origin-destination pair under consideration. Number and size of BDDs is reported for minimum, maximum, 25/50/75/99\% quantiles. Average CPU time per experiment across all 411264 experiments is 2.5s. Compare to \cite[Table~7]{laumanns:14}.}
  \label{tab:random-road}
\end{table}

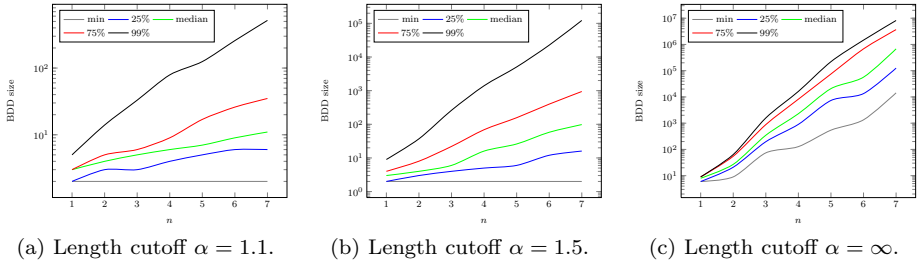
\begin{figure}
  \centering
  \subfloat[][Length cutoff $\alpha=1.1$.]{
  \begin{tikzpicture}[scale=.45]
        \begin{axis}[
          ylabel=BDD size,
          xlabel=$n$,
          ymin=0,
          legend columns=3, legend pos=north west,
          ymode=log
          ]
          \addplot[smooth,color=gray] table [x=n, y index=2] {rr-1.1.csv};
          \addplot[smooth,color=blue] table [x=n, y index=3] {rr-1.1.csv};
          \addplot[smooth,color=green] table [x=n, y index=4] {rr-1.1.csv};
          \addplot[smooth,color=red] table [x=n, y index=5] {rr-1.1.csv};
          \addplot[smooth] table [x=n, y index=6] {rr-1.1.csv};
          \legend{min, 25\%, median, 75\%, 99\%, max}
        \end{axis}

      \end{tikzpicture}
    }
    \hfill
    \subfloat[][Length cutoff $\alpha=1.5$.]{
      \begin{tikzpicture}[scale=.45]
        \begin{axis}[
          ylabel=BDD size,
          xlabel=$n$,
          ymin=0,
          legend columns=3,legend pos=north west,
          ymode=log
          ]
          \addplot[smooth,color=gray] table [x=n, y index=2] {rr-1.5.csv};
          \addplot[smooth,color=blue] table [x=n, y index=3] {rr-1.5.csv};
          \addplot[smooth,color=green] table [x=n, y index=4] {rr-1.5.csv};
          \addplot[smooth,color=red] table [x=n, y index=5] {rr-1.5.csv};
          \addplot[smooth] table [x=n, y index=6] {rr-1.5.csv};
          \legend{min, 25\%, median, 75\%, 99\%, max}
        \end{axis}
      \end{tikzpicture}
    }
    \hfill
    \subfloat[][Length cutoff $\alpha=\infty$.]{
      \begin{tikzpicture}[scale=.45]
        \begin{axis}[
          ylabel=BDD size,
          xlabel=$n$,
          ymin=0,
          legend columns=3,legend pos=north west,
          ymode=log
          ]
          \addplot[smooth,color=gray] table [x=n, y index=2] {rr-inf.csv};
          \addplot[smooth,color=blue] table [x=n, y index=3] {rr-inf.csv};
          \addplot[smooth,color=green] table [x=n, y index=4] {rr-inf.csv};
          \addplot[smooth,color=red] table [x=n, y index=5] {rr-inf.csv};
          \addplot[smooth] table [x=n, y index=6] {rr-inf.csv};
          \legend{min, 25\%, median, 75\%, 99\%, max}
        \end{axis}
      \end{tikzpicture}
    }
  \caption{BDD size comparison across instances of Table~\ref{tab:random-road}..}
  \label{fig:rr-bdd-plots}
\end{figure}

We can apply Lemma~\ref{lem:something-with-guaranteed-size} to this
problem as follows: For fixed $s$ and $t$ and each $\alpha$ the set
$\mathcal{T}\mathcal{M}^\leq_\alpha$ contains the minimal sets of
edges that need to be removed to destroy all shortest $st$-paths of
length at most $\alpha$ (sometimes called $\alpha$-length-bounded
cuts~\cite{baier-erlebach-hall-koehler-schilling-skutella:06}). Clearly,
each such set is a subset of some minimal $st$-cut. Hence, if the
incidence matrix of all minimal $st$-cuts has bandwidth at most $k$,
so do all incidence matrices for the clutters
$\mathcal{T}\mathcal{M}^\leq_\alpha$. Note that for the incidence
matrix of all minimal $st$-cuts to have bandwidth at most $k$, in
particular the cardinality of each cut must be $\leq
k$. Lemma~\ref{lem:bandwidth-limited-cuts} shows a class of such
graphs.

\begin{lemma}\label{lem:bandwidth-limited-cuts}
  Let $G_1,\dots,G_l$ be graphs for which the incidence matrices of
  minimal $s_i-t_i$-cuts ($s_i,t_i\in N(G_i)$) all have bandwidth at
  most $k$. Let $v$ be a new node, i.e. $v\notin\bigcup_iN(G_i)$. Then
  all graphs $G=(N,E)$ of the form $N=\{v\}\cup\bigcup_i N(G_i)$,
  $E=E_v\cup\bigcup_iE(G_i)$ with $E_v=\bigcup_iE^i_v$ such that
  $E_v^i\subset\{(x,v)\::\: x\in N(G_i)\}$ and $|E^i_v|\leq 1$ have
  incidence matrices of minimal $st$-cuts (for all $s,t\in N$) with
  bandwidth at most $k$.
\end{lemma}

\begin{proof}
  If $s$ and $t$ are in an original graph $G_i$ then the statement is
  true by assumption. If $s\in G_i$ and $t=v$ then all minimal cuts
  are either the edge $(v,x)\in E^i_v$, or $s-x$-cuts in $G_i$, so
  have bandwidth at most $k$. Otherwise there is a path from $s$ to
  $t$ in which $v$ is a separating node, so cuts either separate $s$
  and $v$ or $v$ and $t$ and the previous case applies.
\end{proof}

This recursive construction yields tree-like graphs composed of
smaller bandwidth-bounded subgraphs and connecting nodes that are
separators.

We conjecture that the pre-disaster planning problem has polynomial-size exact MIP reformulation if the line graph of $G=(V,E)$ has
logarithmically bounded pathwidth $k\in\mathcal{O}(\log(|V|+|E|))$, by exploiting few vertex cuts in the line graph and translating them to few edge cuts in the original graph, but refer this to further work.



From a practical perspective, we obtained very small BDDs for the
relatively sparse graphs arising in road networks using the following
heuristic. For each $\alpha\in\Crit(f_{\text{SP}})$ use the following
edge ordering: Count the number of occurrences of edge $e$ in the sets
of $\MTP_\alpha(f_{\text{SP}})$ and sort them by increasing value. The
intuition is that edges which are in no element of
$\MTP_\alpha(f_{\text{SP}})$ are `don't care' edges for every
scenario, and will thus not require introduction of any node in the
BDD, while putting central edges (participating in many minimal
scenarios) at the end will not lead to many branching decisions in the
BDD until the very last layers. Furthermore, these edges will be
included to complete many scenarios, so when merging isomorphic
sub-dags in the BDD construction it is good to find them in the bottom
layers. The ordering obtained will be similar, at least for graphs
with few central nodes, to a heuristic low-bandwidth ordering obtained
by the Cuthill-McKee algorithm (\cite{cuthill-mckee:69}). The circuit
system minor equivalence check amounts to a direct matrix minor
comparison, since we assume that the circuit system is explicitly
given by the list of minimal scenarios.

This allows us to give an exact formulation of the pre-disaster
planning problem of~\cite{peeta-salman-gunnec-viswanath:10} as a MIP
of very small size, an order of magnitude smaller than in the recent
work
of~\cite{prestwich-laumanns-kawas:13,laumanns:14}. Table~\ref{tab:part-vs-bdd-bundles}
shows the MIP sizes for the Istanbul instance, and
Table~\ref{tab:random-road} shows number of BDDs and sizes for
randomly constructed networks of similar density, as done
in~\cite{laumanns:14} (see also~\cite{masucci-smith-crooks-batty:09}
for the specifics on sampling random connected graphs). Again, the BDD
sizes directly correspond to MIP sizes.

We remark that the partitioning scheme and `molded distribution'
concept of~\cite{prestwich-laumanns-kawas:13,laumanns:14} can be seen
as a special case of our construction where all BDDs are simply paths,
not arbitrary DAGs.

\subsection{Independence Systems with failing elements}
\label{sec:independence-system}
Let $\mathcal{I}=(E,S)$ with $S\subseteq 2^E$ be an independence
system, i.e. $\forall S_1\in S: S_2\subseteq S_1\Rightarrow S_2\in
S$. Let $w:S\to\R$ with $w(S_i)=\sum_{e\in S_i}w_e$ with $w_e\geq 0$
be a nonnegative linear weight function. Consider a setting where
elements $e\in E$ can fail independently at random.  Let $x_e\in\{0,1\}$
be decision variables such that the element failure probabilities are
$p_e(x_e)=
\begin{cases}
  p_e^1&x_e=1\\
  p_e^0&x_e=0.
\end{cases}$

Consider the 2-stage stochastic programming problem of the
form~\eqref{eq:2-stage-stoch} where for a scenario of failing edges
$\xi\subseteq E$ the inner optimization problem is given by
$f(\xi)=\max_{S_i\cap\xi\in S}w(S_i)$, i.e. by optimizing $w$ over the
restriction of $\mathcal{I}$ to $\xi$. Then this setting covers the
following problems:
\begin{itemize}
\item Expected maximum weight forests in a graph $G=(V,E)$\\
  The independence system considered is the graphic matroid with
  ground set elements $E$, non-negative edge weights $w_e$. Edges can
  fail (independently) with some failure probability that can be
  influenced by decisions $x_e$. The goal is to minimize or maximize
  the expected weight of the maximal forest under edge decisions.
\item Expected stability number in a graph $G=(V,E)$\\
  The independence system considered is the set of stable sets in $G$.
\item Matroid Steiner Problems~\cite{colbourn-pulleyblank:89}
\end{itemize}

As mentioned before, BDD construction is output-polynomial in the graphic matroid. We point out that in all of the cases we believe that relevant graph classes with bounded BDD width can be characterized, but each may be a separate and intricate topic of investigation.

\subsection{Stochastic Flow Network Interdiction}
\label{sec:flow-capacity}

The flow value function of a capacitated network flow problems is
monotone wrt. the network capacities. Consider a setting where arc
capacity failures (or reductions) occur randomly. Decisions allow us
to increase or decrease reliability of the links. If we assume
complete failure of arcs, the fundamental
theorem of network flow theory about path decomposition of flows lets
us link this problem to a classic maximum flow problem. It is more
challenging to consider the case of (discrete) arc capacity changes;
see~\cite{evans:76} and papers citing that for structural statements
on the lattice of flow distributions under capacity changes.

\begin{figure}
  \centering
  \subfloat[][\texttt{SNIP(IB)-4x9}: $38$ nodes, $67$ arcs ($24$ targetable), interdiction success rate $75\%$. Network has $10$ different flow values; $\sum$ BDD sizes: $207$ nodes, construction time $3.3$s ($3969$ max flow oracle calls), MIP size $1771\times 434$]{
    \resizebox{\textwidth}{!}{
\begin{tikzpicture}
  \begin{axis}[compat=1.3,
    title={investment/expected flow plot},
    scale only axis,
    axis y line*=left,
    xlabel={interdiction budget},
    ymin=0,
    ylabel={expected flow}]
    \addplot+[blue] 
    table [x=budget, y=obj, col sep=semicolon] {snip-4x9-timings.csv};
    \label{fig:plot:snip49:flow}
  \end{axis}
  
  \begin{axis}[compat=1.3,
    scale only axis,
    axis y line*=right,
    axis x line=none,
    ymin=0,
    legend pos=outer north east,
    ylabel={CPLEX 12.5 time (in s)}]
    \addlegendimage{/pgfplots/refstyle={fig:plot:snip49:flow}}
    \addlegendentry{expected flow}
    \addplot+[green]
    table [x=budget, y=time, col sep=semicolon]
    {snip-4x9-timings.csv};
    \addlegendentry{time}
  \end{axis}
\end{tikzpicture}
\qquad
\begin{tikzpicture}
  \begin{axis}[compat=1.3,
    title={log/log investment/yield plot},
    scale only axis,
    xmode=log,
    ymode=log,
    axis y line*=left,
    xlabel={interdiction budget},
    ymin=0,
    ylabel={expected flow}]
    \addplot+[blue] 
    table [x=budget, y=obj, col sep=semicolon] {snip-4x9-timings.csv};
    \label{fig:plot:snip49:flow}
    \addlegendentry{expected flow}
  \end{axis}
  
\end{tikzpicture}
    }
  }

  \subfloat[][\texttt{SNIP(IB)-7x5}:
  $37$ nodes, $72$ arcs ($22$ targetable), interdiction success rate $75$\%. Network has $19$ different flow values;
  $\sum$ BDD sizes: $2302$ nodes, construction time $5.9$s ($61988$ max flow oracle calls), MIP size $22867\times 5133$]{
    \resizebox{\textwidth}{!}{
\begin{tikzpicture}
  \begin{axis}[compat=1.3,
    title={investment/expected flow plot},
    scale only axis,
    axis y line*=left,
    xlabel={interdiction budget},
    ymin=0,
    ylabel={expected flow}]
    \addplot+[blue] 
    table [x=budget, y=obj, col sep=semicolon] {snip-7x5-timings.csv};
    \label{fig:plot:snip75:flow}
  \end{axis}
  
  \begin{axis}[compat=1.3,
    scale only axis,
    axis y line*=right,
    axis x line=none,
    ymin=0,
    legend pos=outer north east,
    ylabel={CPLEX 12.5 time (in s)}]
    \addlegendimage{/pgfplots/refstyle={fig:plot:snip75:flow}}
    \addlegendentry{expected flow}
    \addplot+[green]
    table [x=budget, y=time, col sep=semicolon]
    {snip-7x5-timings.csv};
    \addlegendentry{time}
  \end{axis}
\end{tikzpicture}
\qquad
\begin{tikzpicture}
  \begin{axis}[compat=1.3,
    title={log/log investment/yield plot},
    scale only axis,
    xmode=log,
    ymode=log,
    axis y line*=left,
    xlabel={interdiction budget},
    ymin=0,
    ylabel={expected flow}]
    \addplot+[blue] 
    table [x=budget, y=obj, col sep=semicolon] {snip-7x5-timings.csv};
    \label{fig:plot:snip75:flow:loglog}
    \addlegendentry{expected flow}
  \end{axis}
  
\end{tikzpicture}
    }
  }


  \subfloat[][\texttt{SNIP(IB)-10x10}:
  $10\times 10$ grid,  35\% interdictable arcs, vertical arc orientation unif. random up/down $50:50$, capacities $10..100$ uniformly random multiples of $10$, $75$\% interdiction success rates.]{
    \resizebox{\textwidth}{!}{
\begin{tikzpicture}
  \begin{axis}[compat=1.3,
    title={investment/expected flow plot},
    scale only axis,
    axis y line*=left,
    xlabel={interdiction budget},
    ymin=0,
    ylabel={expected flow}]
    \addplot+
    table [ x index={0}, y index={1}, col sep=space]  {snip-10x10.35-timings-marco.csv};
    \label{fig:plot:snip10x10.35}
  \end{axis}

  \begin{axis}[compat=1.3,
    scale only axis,
    axis y line*=right,
    axis x line=none,
    scaled ticks=false,
    ymin=0,
    ymax=375000,
    legend pos=outer north east,
    ylabel={CPLEX 12.6.2 time (in s)}]
    \addlegendimage{/pgfplots/refstyle={fig:plot:snip10x10.35}}
    \addlegendentry{expected flow}
    \addplot+[green]
    table [x index={0}, y index={2}, col sep=space]
    {snip-10x10.35-timings-marco.csv};
    \addlegendentry{time}
  \end{axis}

\end{tikzpicture}

\qquad
\begin{tikzpicture}
  \begin{axis}[compat=1.3,
    title={log/log investment/yield plot},
    scale only axis,
    xmode=log,
    ymode=log,
    axis y line*=left,
    xlabel={interdiction budget},
    ymin=0,
    ylabel={expected flow}]
    \addplot+[blue] 
    table [x index={0}, y index={1}, col sep=space] {snip-10x10.35-timings-marco.csv};
    \label{fig:plot:snip101035:flow:loglog}
    \addlegendentry{expected flow}
  \end{axis}
  
\end{tikzpicture}

    }
  }
    
  \caption{SNIP(IB) Instances of~\cite{cormican-morton-wood:98}. Computations performed using CPLEX on a single node 16-core Intel(R) Xeon(R) CPU E5-2698 v3 @ 2.30GHz.}
  \label{fig:snip-instances}
\end{figure}
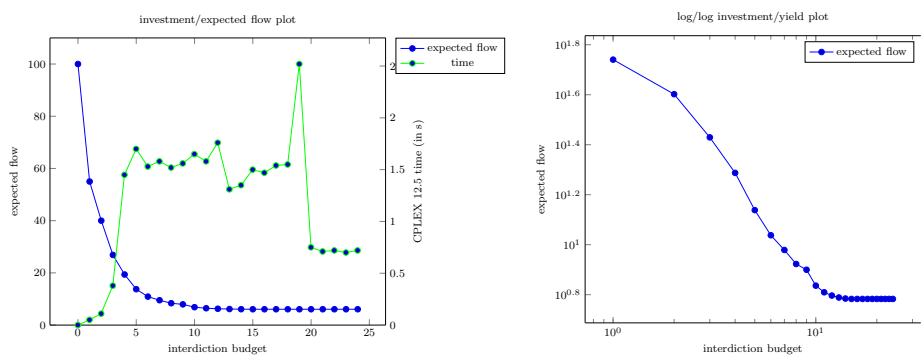
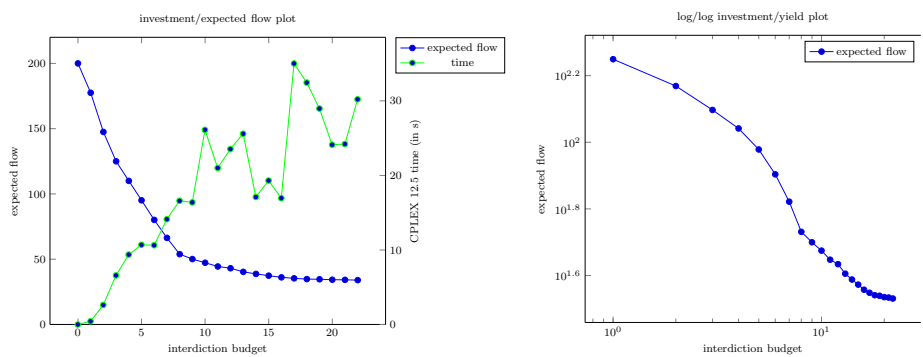
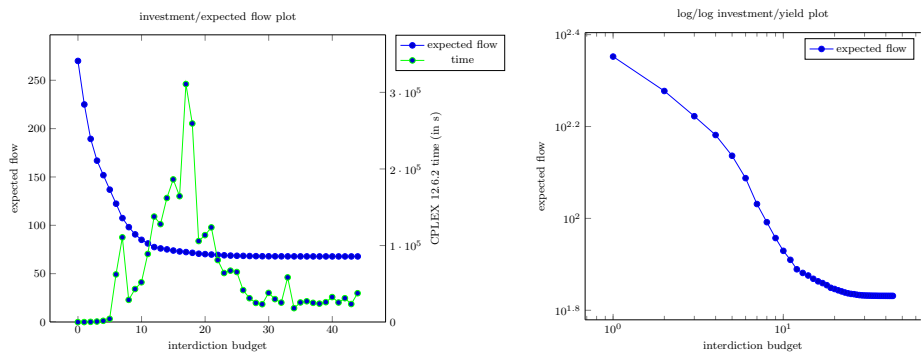

In particular, the network interdiction problem SNIP(IB) introduced
in~\cite{cormican-morton-wood:98} fits this framework:
\begin{definition}[stochastic network interdiction problem SNIP(IB)]
  Let $G=(V,A)$ be a (directed) graph, $s,t\in V$ a source and a sink
  node, and arc capacities $\{c_e\}_{e\in A}$ be given. For each arc
  $e\in A$ let a probability $p_e$ of an interdiction attempt being
  successful be given, where success means that the arc capacity
  becomes $0$. The problem SNIP(IB) for $(G,\{p_e\}_{e\in A},s,t)$ is
  then the problem of selecting a set of arcs to interdict, such that
  the expected maximum flow in the remaining network is minimized.
\end{definition}

To solve an SNIP(IB) problem using our framework we need to compute
the set of all $st$-flows arising under edge removal from the graph
$G$. By the max flow min cut theorem the number of different flow
values appearing depends on the number of possible cut values, so in
particular a rough estimate shows that if the number of different arc
capacities is $k$ there are at most $\binom{|E|}{k}$ possible cut
values. Having a limited number of arc capacities is very natural in
many applications where the capacity is determined by a technological
parameter of the arc type, e.g., link speed or pipe diameter.

For the purpose of showing practicality of our approach we resort to
computing the sets of all minimal edge sets whose removal reduces the
maximal $st$-flow in the residual graph to less than $\alpha$ using a
technique successfully applied to determining minimal cut sets in
metabolic flux
networks~\cite{haus-klamt-stephen:08,ballerstein-haus-kamp-klamt:2011}. In
particular, we solve the max-flow problem using the \texttt{lemon-1.3.1} library~\cite{lemon-lib} in
the membership oracle of the joint generation
procedure~\cite{FK96,cl-jointgen}. In Figure~\ref{fig:snip-instances} we use
the benchmark SNIP(IB)-instances of~\cite{cormican-morton-wood:98}. 
Due to the fact that the MIP only depends on the network structure and the set of interdictable arcs we can seamlessly change the budget value or interdiction success probabilities without re-generating the MIP. This separation of structure and data is a major advantage to other methods that take these parameters into account throughout the solution procedure like~\cite{janjarassuk-linderoth:08}, and makes it possible to do a parameter study accounting for all conceivable budget values.

In summary, we find that we obtain MIPs that allow us to solve the exact reformulation, rather than resorting to solving approximations, and for all conceivable budgets, but that current MIP solvers -- we've also done limited trials with FICO Xpress, Gurobi and SCIP -- struggle with the particular problem structure more than we expected. The formulation with few, binary variables coupling many continuous variables in equations seems like a natural fit for decomposition approaches, which we are planning to investigate in future research.






\begin{acknowledgements}
  This work was partially supported by EU-FP7-ITN 289581 `NPlast'
  (while the second author was at ETH Zürich) and EU-FP7-ITN 316647 `MINO'
  (while the first author was at ETH Zürich).
\end{acknowledgements}

\bibliographystyle{plain}
\bibliography{scenario-bundling}
\bigskip

\end{document}